\documentclass[11pt,A4]{article}
\usepackage{amsmath,amsthm,amsfonts,amssymb,amscd, amsxtra,color}
\usepackage[active]{srcltx}
\usepackage{url}
\usepackage{cite}
\usepackage{multirow}
\usepackage[ruled]{algorithm2e}
\usepackage[margin=2.5cm,nohead]{geometry}
\newtheorem{theorem}{Theorem}
\newtheorem{lemma}{Lemma}
\newtheorem{corollary}{Corollary}
\newtheorem{proposition}{Proposition}
\newtheorem{remark}{Remark}
\newtheorem{definition}{Definition}
\newtheorem{example}{Example}

\usepackage{amssymb,amsfonts,amstext,amsmath}
\usepackage{xcolor}
\usepackage{graphicx,geometry}

\usepackage{amssymb}

 \usepackage{soul,color}

\newtheorem{strategy}{Strategy}

\usepackage{geometry}
\usepackage{color}
\usepackage{float}
\usepackage{textcomp}
\usepackage{amsthm}
\usepackage{amsmath}
\usepackage{amssymb}

\usepackage{changes}

\usepackage{mathtools}

\begin{document}

\title{Iteration-Complexity of the Subgradient  Method  on Riemannian Manifolds with Lower Bounded Curvature} 

\author{
O. P. Ferreira \thanks{IME/UFG, Avenida Esperan\c{c}a, s/n, Campus Samambaia,  Goi\^ania, GO, 74690-900, Brazil (e-mails: {\tt orizon@ufg.br},  {\tt mauriciosilvalouzeiro@gmail.com},  {\tt lfprudente@ufg.br}).}
\and
 M. S. Louzeiro \footnotemark[1]
\and
L. F. Prudente \footnotemark[1]
}
\maketitle
\begin{abstract}
The subgradient  method  for convex optimization problems on complete  Riemannian manifolds with lower bounded sectional curvature is  analyzed in this paper.   Iteration-complexity bounds  of the subgradient method  with  exogenous step-size and Polyak's step-size   are stablished, completing and improving recent results on the subject.
\\[2mm]

\end{abstract}
\section{Introduction}
In this paper we  consider the   subgradient method to solve the optimization problem defined by:
\begin{equation} \label{eq:OptP}
\min \{ f(p) ~:~   p\in \mathcal{M}\},
\end{equation}
where the constraint set $\mathcal{M}$ is  endowed  with a structure of a  {\it complete   Riemannian manifold with lower bounded curvature}  and $f:\mathcal{M}\to \overline{\mathbb{R}}$ is a {\it convex function},  where $\overline{\mathbb{R}}=\mathbb{R}\cup\{+\infty\}$ denotes the extended real set numbers. 

It is well known that non-convex problems can be transformed into convex by introducing a suitable metric.  As a consequence, this technique can be exploited in order to find global minimizers \cite{BentoMelo2012, BentoFerreiraOliveira2015, FerreiraCPN2006, Rapcsak1997} { and to reduce the  iteration-complexity of finding such solutions \cite{BentoFerreiraMelo2017,ZhangReddiSra2016}.
Furthermore, many optimization problems are naturally posed on Riemannian manifolds which have a specific underlying geometric and algebraic structure that can be exploited to greatly reduce the cost of obtaining solutions.  For instance, in order to take advantage of the Riemannian geometric  structure, it is preferable to treat certain constrained optimization problems as problems for finding singularities of gradient vector fields  on  Riemannian manifolds rather than using Lagrange multipliers  or projection methods; see \cite{Luenberger1972,   Smith1994, Udriste1994}.   Accordingly, constrained optimization problems are viewed as unconstrained ones from a Riemannian geometry point of view. Besides, Riemannian geometry also opens up new research directions  that aid in developing competitive algorithms;  see \cite{EdelmanAriasSmith1999, NesterovTodd2002, Smith1994}. For this purpose,  extensions of concepts and techniques of optimization from Euclidean  space to Riemannian context have been quite frequent in recent years.  Papers dealing with this subject include, but are not limited to \cite{FerreiraLouzeiroPrudente2018, LiMordukhovichWang2011,  LiYao2012,  WangLiWangYao2015,WangLiYao2015,Manton2015}.

The subgradient method is a very simple algorithm for solving convex optimization  problems and,  besides  it is the departure point for many other more sophisticated and efficient algorithms,   including   $\epsilon$-subgradient methods,  bundle methods and cutting-plane algorithm; see \cite{CorreaLemarechal1993} for a  comprehensive study on this subject.  The subgradient method was originally developed by Shor and others  in the 1960s and 1970s  and since of then,  it and its variants have been applied to a far wider variety of problems in optimization theory; see\cite{Poljak1978, Goffin2012}.  In order to deal with non-smooth convex optimization problems on  complete Riemanian manifolds with non-negative sectional curvature, \cite{FerreiraOliveira1998}  extended   and analyzed the subgradient method  which,  as in the Euclidean context,  is quite simple and possess nice convergence properties.  After this pioneering work,   the subgradient method in the Riemannian setting has been studied in different  contexts; see, for instance, \cite{BentoMelo2012, WangLiWangYao2015, WangLiYao2015, GrohsHosseini2016, BentoCruzNeto2013}. In \cite{BentoMelo2012} the subgradient method was introduced to solve  convex feasibility  problems on complete Riemannian manifolds with non-negative sectional curvatures, and recently in \cite{WangLiWangYao2015, WangLiYao2015}   this method has been analyzed   in manifolds with lower bounded sectional curvatures  and significant  improvements were introduced. More recently,  an asymptotic analysis of the subgradient method  with exogenous step-size and dynamic step-size for convex optimization was considered  in the context of manifolds with lower bounded sectional curvatures, see \cite{Wang2018}.

In this paper we  establish an iteration-complexity bound of the subgradient method  with  exogenous step-size and Polyak's step-size,  for convex optimization problems on complete  Riemannian manifolds with lower bounded sectional curvatures.  Our results  increase the range of applicability of the method compared to the respective results  obtained in \cite{BentoFerreiraMelo2017, ZhangReddiSra2016,  ZhangSra2016}. 
 Moreover, in the asymptotic analysis with  exogenous step-size, we do not assume that the solution set is nonempty, completing the  result of } \cite[Theorem 3.1]{Wang2018}. It should be noted that our analysis use a recently inequality stablished in  \cite{WangLiWangYao2015, WangLiYao2015}.

This paper is organized as follows. Section~\ref{sec:aux} presents some definitions and  preliminary results  related to the Riemannian geometry  that are important to  our study. In Section~\ref{sec:subgradient}, we obtain  iteration-complexity bounds  and other convergence results of the subgradient method with exogenous step-size and Polyak's step-size. In the Section~\ref{sec:num.exem} we
use  the convex feasibility problem for numerically illustrate the results on complexity-iteration bounds of Section~\ref{sec:subgradient}. The last section contains a conclusion.

\section{Notations and basic concepts} \label{sec:aux}
In this section, we introduce some fundamental  properties and notations about Riemannian geometry. These basics facts can be found in any introductory book on Riemannian geometry; see for example, \cite{doCarmo1992, Sakai1996}. We also recall the  definitions  of convexity of function and   Lipschitz continuity in the Riemannian setting   and present some basic properties related to these concepts that will be essential for the analyses of the  subgradient method in the next section.

{\it In this paper $\mathcal{M}$ is endowed  with a structure of a   complete   Riemannian manifold with lower bounded curvature. Throughout the paper  we also assume that the sectional curvature  of $\mathcal{M}$ is bounded below by $\kappa<0$}.  We denote by $T_p\mathcal{M}$ the {\it tangent space} of a Riemannian manifold $\mathcal{M}$ at $p$.  The corresponding norm associated to the Riemannian metric $\langle \cdot ,  \cdot \rangle$ is denoted by $\|  \cdot \|$. We use $\ell(\alpha)$ to denote the length of a piecewise smooth curve $\alpha:[a,b]\rightarrow \mathcal{M}$. The Riemannian  distance  between $p$ and $q$   in a finite dimensional Riemannian manifold $\mathcal{M}$ is denoted  by $d(p,q)$,  which induces the original topology on $\mathcal{M}$, namely,  $(\mathcal{M},d)$ is a complete metric space  where bounded and closed subsets are compact.  Let $\nabla$ be the Levi-Civita connection associated to $(\mathcal{M}, \langle \cdot ~, ~ \cdot \rangle)$.  A vector field $V$ along $\gamma$ is said to be {\it parallel} iff $\nabla_{\gamma^{\prime}} V=0$. If $\gamma^{\prime}$ itself is parallel we say that $\gamma$ is a {\it geodesic}.  Since the geodesic equation $\nabla_{\ \gamma^{\prime}} \gamma^{\prime}=0$ is a second order nonlinear ordinary differential equation, then the geodesic $\gamma=\gamma _{v}( \cdot ,p)$ is determined by its position $p$ and velocity $v$ at $p$. It is easy to check that $\|\gamma ^{\prime}\|$ is constant. The restriction of a geodesic to a  closed bounded interval is called a {\it geodesic segment}. A geodesic segment joining $p$ to $q$ in $\mathcal{M}$ is said to be {\it minimal} if its length is equal to $d(p,q)$. Hopf-Rinow's theorem asserts that any pair of points in a  complete Riemannian  manifold $\mathcal{M}$ can be joined by a (not necessarily unique) minimal geodesic segment.  Due to  the completeness of the Riemannian manifold $\mathcal{M}$, the {\it exponential map} $\exp_{p}:T_{p}\mathcal{M}\to \mathcal{M} $ can be  given by $\exp_{p}v\,=\, \gamma _{v}(1,p)$, for each $p\in \mathcal{M}$.   We proceed by recalling  some concepts and basic properties  about  convexity   in the  Riemannin context.   For more details   see, for example, \cite{Udriste1994,  Rapcsak1997}.  For any two points $p,q\in\mathcal{M}$, $\Gamma_{pq}$ denotes the set of all geodesic segments  $\gamma:[0,1]\rightarrow\mathcal{M}$ with $\gamma(0)=p$ and $\gamma(1)=q$. The closed metric ball in $\mathcal{M}$ centered at the point $p\in\mathcal{M}$ with radius $r>0$ is denoted by $B[p,r]$. Let $\Omega$  be a subset of $\mathcal{M}$. We use  $\emph{int}~\Omega$ to denote the interior of $\Omega$.  A function  $f:\mathcal{M}\rightarrow\overline{\mathbb{R}}$ is said to be proper if  its domain $\emph{dom} f=\{p\in\mathcal{M}:f(p)\neq+\infty\}$ is nonempty, where $\overline{\mathbb{R}}=\mathbb{R}\cup\{+\infty\}$ denotes the extended real set numbers. We use $\Gamma^{f}_{pq}$ to denote the set of all $\gamma\in\Gamma_{pq}$ such that $\gamma\subseteq\emph{dom} f$.  A nonempty subset  $\Omega \subset \mathcal{M}$ is said to be {\it weakly convex} if, for any $p,q\in \Omega$, there is a minimal geodesic segment joining   $p$ to $q$ and it is in $\Omega$.  A  proper function $f: \mathcal{M} \to\overline{\mathbb{R}}$  is said to be {\it convex} on $\mathcal{M}$ if $\emph{dom} f$ is weakly convex and for any $p,q\in\emph{dom} f$ and $\gamma\in\Gamma^{f}_{pq}$ the composition $f\circ\gamma:[0, 1]\to \overline{\mathbb{R}}$ is a convex function on $[0,1]$ i.e.,  
	$$
	f\circ\gamma(t)\leq(1-t)f(p)+tf(q),\qquad \forall~ t\in[0,1], 
	$$
see \cite{WangLiWangYao2015}.  The {\it subdifferential} of a convex function $f:\mathcal{M}\rightarrow\overline{\mathbb{R}}$ at  $p\in\emph{dom} f$ is defined by
\begin{equation} \label{eq:defSugrad}
\partial f(p):=\left\{s\in T_p\mathcal{M}: f(q)\geq f(p)+\left\langle s,\gamma'(0)\right\rangle, \quad \forall ~y\in\emph{dom} f,~\gamma\in \Gamma^{f}_{pq}\right\}.
\end{equation}

We remark that  the subdiffential set  $\partial f(p)$ is nonempty in all  \st{at} $p\in  \emph{int}\, \emph{dom} f$; see  \cite[Proposition 2.5]{WangLiWangYao2015}. {\it In this paper all functions  $f: \mathcal{M} \to\overline{\mathbb{R}}$  is assumed  to be  convex and lower semicontinuous  on $\mathcal{M}$}. The following result is also proved  in \cite[Proposition 2.5]{WangLiWangYao2015}.

\begin{proposition}\label{prop:subglim}
Let $\{p_k\}\subset \mathcal{M}$ a bounded sequence. If the sequence $\{s_k\}$ is such that $s_k\in\partial f(p_k)$, for each $k\in\mathbb{N}$, then $\{s_k\}$ is also bounded.
\end{proposition}

 The following lemma plays  an important role in the next sections. Its proof will be omitted here, but it can be obtained, with some minor technical adjustments, by using the Toponogov's theorem~\cite[p.161, Theorem 4.2]{Sakai1996} and following the ideas of~\cite[Lemma~3.2]{WangLiWangYao2015}; see also \cite{WangLiYao2015}.

\begin{lemma} \label{lemli}
Let $p\in\emph{int}\, \emph{dom} f$, $0\neq s\in \partial f(p)$, and let $\gamma:[0,+\infty)\rightarrow \mathcal{M}$ be the geodesic defined by $ {\gamma}(t)=\mbox{exp} _{p}\left(-ts/\left\|s\right\|\right).$
Then, for any $t\in[0,\infty)$ and $q\in \emph{dom} f$ there holds
\begin{multline*}
\cosh(\sqrt{|\kappa|}d(\gamma(t),q))\leq \cosh(\sqrt{|\kappa|}d(p,q))+\\
\sqrt{|\kappa|}\cosh(\sqrt{|\kappa|}d(p,q))\sinh(t\sqrt{|\kappa|})\left[\frac{t}{2}-\frac{\tanh(\sqrt{|\kappa|}d(p,q))}{\sqrt{|\kappa|}d(p,q)} \frac{f(p)-f(q)}{\left\|s\right\|}\right]
\end{multline*}
and, consequently, the following inequality  holds 
$$
d^2({\gamma}(t),q)\leq d^2(p,q) + \frac{\sinh\left(\sqrt{|\kappa|}t\right)}{\sqrt{|\kappa|}} \left[\frac{\sqrt{|\kappa|}d(p,q)}{\tanh\left(\sqrt{|\kappa|}d(p,q)\right)}t-\frac{2}{\left\|s\right\|}\left(f(p)-f(q)\right)\right]. 
$$
\end{lemma}

The next   concept will be useful in the analysis of the sequence generated by the subgradient method.

\begin{definition} \label{def:QuasiFejer}
A sequence $\{y_k\}$ in the complete metric space $(\mathcal{M},d)$ is quasi-Fej\'er convergent to a set $W\subset \mathcal{M}$ if, for every $w\in W$, there exists a sequence $\{\epsilon_k\}\subset\mathbb{R}$ such that $\epsilon_k\geq 0$, $\sum_{k=1}^{\infty}\epsilon_k<+\infty$, and
$$
d^2(y_{k+1},w)\leq d^2(y_k,w)+\epsilon_k, \qquad \forall~k.
$$
\end{definition}

The main property of the  quasi-Fej\'er convergent sequence is stated  in the next result, and its proof is similar to the one proved in \cite{burachik1995full}  by replacing the Euclidean  by the Riemannian distance.

\begin{theorem}\label{teo.qf}
Let $\{y_k\}$ be a sequence in the complete metric space $(\mathcal{M},d)$. If $\{y_k\}$ is quasi-Fej\'er convergent to a nomempty set $W\subset \mathcal{M}$, then $\{y_k\}$ is bounded. If furthermore, a cluster point $y$ of $\{y_k\}$ belongs to $W$, then $\lim_{k\rightarrow\infty}y_k=y$.
\end{theorem}

We end this section by recalling the concept of  Lipschitz continuity of a function.  A  proper function $f: \mathcal{M} \to\overline{\mathbb{R}}$ is said to be {\it Lipschitz continuous with  constant $\tau \geq 0$ }  in $\Omega \subset \mathcal{M} $ if $ \left |f(p)- f(q)\right |\leq \tau\, d(p,q)$, for any   $ p,q\in \Omega$.

\section{ Iteration-Complexity of the Subgradient Method}\label{sec:subgradient}
In this section, we state the    Riemannian subgradient method   to solve \eqref{eq:OptP} and the strategies for  choosing  the step-size that will be used in our analysis. Let  $f: \mathcal{M}\to  \overline{\mathbb{R}}$  be   convex function, $\Omega^*$  be the {\it solution set} of   the problem \eqref{eq:OptP} and $f^*\coloneqq \inf_{x\in \mathcal{M}}f(x)$   be the  {\it optimum value}  of $f$. {\it In  our analysis we do not assume that   $\Omega^*$  is nonempty,  except   when explicitly stated}.  The  statement  of    {\it Riemannian  subgradient algorithm} to solve the problem \eqref{eq:OptP} is as follows.\\
\begin{algorithm}[H]
\begin{description}
\item[ Step 0.] Let $p_0 \in \emph{int}\, \emph{dom} f$. Set $k=0$.
\item[ Step 1.] If $s_k=0$, then {\bf stop}; otherwise,  choose a step-size $t_k>0$, $s_{k}\in \partial f(p_k)$ and compute
\begin{equation} \label{eq:SubGradMethod}
p_{k+1}:=\exp_{p_{k}}\left(-t_{k} \frac{s_{k}}{\|s_{k}\|}\right);
\end{equation}
\item[ Step 2.]  Set $k\leftarrow k+1$ and proceed to  \textbf{Step~1}.
\end{description}
\caption{Subgradient algorithm in a Riemanian manifold  $\mathcal{M}$}
\label{alg:subgradient}
\end{algorithm}
In the following we present two  different strategies for  choosing  the step-size $t_k>0$ in Algorithm~\ref{alg:subgradient}.  
\begin{strategy}[Exogenous step-size]\label{Exogenous.Step} 
\begin{equation}\label{step-size}
t_k> 0, \qquad \qquad  \sum_{k=0}^{\infty}t_k=+\infty, \qquad \qquad  \sigma:=\sum_{k=0}^{\infty} t_k^2<+\infty. 
\end{equation}
\end{strategy}
The step-size in Strategy~\ref{Exogenous.Step} have been used   in several paper  for analyzing subgradient method; see, for example,  \cite{FerreiraOliveira1998, CorreaLemarechal1993, Wang2018}. 
\begin{strategy}[Polyak's  step-size]\label{Poliak.Step} 
 Assume that  $p_0 \in \emph{int}\, \emph{dom} f$, $\Omega^* \neq \varnothing$ an set 
\begin{equation}\label{StepsizePolyak}
t_k=\alpha\frac{f(p_k)-f^*}{\left\|s_k\right\|},\qquad 0< \alpha<2\frac{\tanh\left(\sqrt{|\kappa|}d_0 \right)}{ \sqrt{|\kappa|}d_0}, \qquad  d_0:=d(p_0, \Omega^*), 
\end{equation}
where  $d(p_0, \Omega^*):=\inf\{ d(p_0, q);~q\in \Omega^*\}>0$.
\end{strategy}
This step-size  in  Strategy~\ref{Poliak.Step}  was introduced in \cite{Poljak1978}  and has been used in \cite{BentoFerreiraMelo2017, BentoMelo2012, WangLiWangYao2015 }. 

\begin{remark}
Since the function $ (0,+\infty)\mapsto \tanh(t)/t$ is decreasing,  then given an estimate  $\hat{d}> d_0$ we can chose 
$0< \alpha <2 \tanh(\sqrt{|\kappa|}\hat{d} )/(\sqrt{|\kappa|}\hat{d}))$ in Strategy~\ref{Poliak.Step}. 
For Riemannian manifold with non-negative curvature, the  second inequality in \eqref{StepsizePolyak} holds for all   $\kappa <0$. Due to $\lim_{t\to 0}\tanh(t)/t=1$,  letting  $\kappa$ goes to $0$, we can chose $0<\alpha<2$.
\end{remark}

{\it From now on we assume that the sequence $\{p_k\}$ generated by   Algorithm~\ref{alg:subgradient}  with the  two above strategies  for choosing the step-size   is well defined and is infinite.}

\begin{remark}
Note that if  $\emph{dom} f=\mathcal{M}$ then  $\partial f(p) \neq \varnothing$, for all $p\in \mathcal{M}$ and, consequently,  the sequence  $\{p_k\}$  is well defined. In \cite[Theorem 3.1]{Wang2018} an asymptotic convergence analysis was established,  under the assumption that  suitable sets  are contained  in $\emph{int}\, \emph{dom} f$ and that the set $\Omega^*$ is nonempty. The author proves that the sequence generated by the Algorithm~\ref{alg:subgradient} is well defined and converges to an element of $\Omega^*$.   It is worth to pointed out that our asymptotic convergence analysis of Algorithm~\ref{alg:subgradient},  with Strategy~\ref{Exogenous.Step} for choosing the step-size,   we do not assume that $\Omega^*$ is nonempty. In this sense,  our results  improve the ones  of \cite[Theorem~3.1]{Wang2018}.  
\end{remark}

\subsection{Subgradient Method with Exogenous Stepsize} \label{Sec:aca}
In this section  we assume that the sequence $\{p_k\}$  is generated by   Algorithm~\ref{alg:subgradient}  with  Strategy~\ref{Exogenous.Step}  for choosing the step-size. To proceed with the analysis of \st{the} Algorithm~\ref{alg:subgradient}  we  need some preliminaries. Firstly we define
\begin{equation} \label{eq;DefOmega}
\Omega:=\left\{q\in \mathcal{M}:~f(q)\leq \inf_kf(p_k)\right\}.
\end{equation}
Note that $\Omega \subset \emph{dom} f$.  It is worth mentioning that, in principle,  the set $\Omega$ can be empty. Our first task is to prove that the sequence  $\{p_k\}$  is bounded. 
\begin{lemma}\label{lem:bounded}
 If $\Omega \neq \varnothing$ then, for each  $q\in \Omega$ there  holds
\begin{equation}\label{def:dqkapa}
d(p_{k+1},q)\leq \frac{1}{\sqrt{|\kappa|}}\cosh^{-1}\left(\cosh \left(\sqrt{|\kappa|}d(p_0,q)\right)e^{\frac{1}{2}\sqrt{\sigma |\kappa|}\sinh\left(\sqrt{\sigma |\kappa|}\right)}\right), 
\end{equation}
for all $ k=0, 1, \ldots.$.
\end{lemma}
\begin{proof}
Applying  first inequality of Lemma~\ref{lemli} with $t=t_k$, $p=p_k$  and  $p_{k+1}=\gamma(t_k)$ and taking into account that  $q\in\Omega$ we  conclude that   
$$
\cosh\left(\sqrt{|\kappa|}d(p_{k+1},q)\right)\leq \cosh\left(\sqrt{|\kappa|}d(p_k,q)\right)\left[1+ |\kappa| t_k^2 \frac{\sinh\left(\sqrt{|\kappa|} t_k\right)}{2 \sqrt{|\kappa|}t_k}\right],  \qquad k=0, 1, \ldots.
$$
Using definition of $\sigma$ in  \eqref{step-size} we have $ t_k \leq \sqrt{\sigma}$, for all $k=0, 1, \ldots$.  Since  the map  $ (0, +\infty) \ni t \mapsto  \sinh(t)/t$ is increasing, it follows from the last inequality  that 
$$
\cosh\left(\sqrt{|\kappa|}d(p_{k+1},q)\right)\leq \cosh\left(\sqrt{|\kappa|}d(p_k,q)\right)\left[1+at_k^2\right], \qquad k=0, 1, \ldots, 
$$
where  $a:= \sqrt{|\kappa|}(\sinh(\sqrt{\sigma|\kappa|}))/(2\sqrt{\sigma})$.  Note  that  the last inequality implies that 
$$
\cosh\left(\sqrt{|\kappa|}d(p_{k+1},q)\right)\leq \cosh\left(\sqrt{|\kappa|}d(p_k,q)\right)e^{at_k^2}, \qquad k=0, 1, \ldots.
$$
Therefore, we have  $\cosh (\sqrt{|\kappa|}d(p_{k+1},q))\leq\cosh(\sqrt{|\kappa|}d(p_0,q))e^{a\sigma}$, which is equivalent to \eqref{def:dqkapa} and the proof is concluded.
\end{proof}
In the next result we apply Lemmas~\ref {lemli} and \ref{lem:bounded} to  derive an inequality that plays an important role in our analysis, which  is a generalization  of the one obtained  in  \cite[Lemma 4.1]{FerreiraOliveira1998}.   In the linear setting, this  inequality is of fundamental importance to analyze the subgradient method; see, for example, \cite{CorreaLemarechal1993}. It is worth noting that  it was obtained in \cite{Yang2010}  for an specific  function, namely, the mean function. For stating the next result, for each $q\in\Omega$,  we define
\begin{equation} \label{eq:cqk}
C_{q,\kappa}:=  \frac{\sinh\left(\sqrt{\sigma\left|\kappa\right|}\right)}{ \sqrt{\sigma\left|\kappa\right|}}\left[1+ \cosh^{-1}\left(\cosh(\sqrt{\left|\kappa\right|}d(p_0,q))e^{\frac{1}{2}\sqrt{\sigma \left|\kappa\right|}\sinh\left(\sqrt{\sigma \left|\kappa\right|}\right)}\right)\right].
\end{equation} 
It is important to note that  $C_{q,\kappa}$ is well defined only  under the assumption $\Omega \neq \varnothing$.
\begin{lemma}\label{lem:MainIneq}
If $\Omega \neq \varnothing$ then, for each  $q\in \Omega$ there  holds
$$
d^2(p_{k+1},q)\leq d^2(p_k,q) + C_{q,\kappa} t_k^2+2\,\frac{t_k}{\left\|s_k\right\|}[f(q)-f(p_k)],  \qquad  s_{k}\in \partial f(p_k), \qquad k=0,1,\ldots.
$$
\end{lemma}
\begin{proof}
Applying  first inequality of Lemma~\ref{lemli} with $t=t_k$, $p=p_k$  and  $p_{k+1}=\gamma(t_k)$, and taking into account that  $q\in\Omega$, we  conclude that  
$$
d^2(p_{k+1},q)\leq d^2(p_k,q) +
 \frac{\sinh\left(\sqrt{|\kappa|}t_k\right)}{\sqrt{|\kappa|}t_k}\left[\frac{\sqrt{|\kappa|}d(p_k,q)}{\tanh\left(\sqrt{|\kappa|}d(p_k,q)\right)} t_k^2+
\frac{2t_k}{\|s_k\|}\left[f(q)-f(p_k)\right]\right], 
$$
for all $k=0,1,\ldots$.   On the other hand,     $t/\tanh(t)\leq 1+t$, for all $t\geq 0$ and the map $ (0, +\infty) \ni t \mapsto  \sinh(t)/t$ is increasing and bounded below by $1$. Thus, taking into account  that $ t_k \leq \sqrt{\sigma}$  and   $f(q)-f(p_k)\leq 0$, for all $k=0,1, \ldots$, we conclude that 
$$
d^2(p_{k+1},q)\leq d^2(p_k,q) + \frac{\sinh\left(\sqrt{\sigma|\kappa|}\right)}{\sqrt{\sigma|\kappa|}} \left[1+ \sqrt{|\kappa|} d(p_k,q)\right] t^2_k+ 
\frac{2t_k}{\|s_k\|}\left[f(q)-f(p_k)\right], 
$$
for all $k=0,1,\ldots$. Therefore,  combining   Lemma~\ref{lem:bounded}  with  \eqref{eq:cqk} the desired inequality follows and the proof is concluded. 
\end{proof}

\begin{remark}
For Riemannian manifold with non-negative curvature, the inequality in Lemma~\ref{lem:MainIneq} holds for all $\kappa<0$.  Since   $\lim_{\kappa \to 0}C_{q,\kappa}=1$,  the inequality of the Lemma~\ref{lem:MainIneq}  merges into  the inequality \cite[Lemma 4.1]{FerreiraOliveira1998}. 
\end{remark}
y
Now we ready to prove the main result of this section.
\begin{theorem}  \label{eq:icomp}
Assume that  $\Omega^* \neq \varnothing$ and    $f: \mathcal{M}\to  \overline{\mathbb{R}}$  is   Lipschitz continuous with  constant $\tau \geq 0$.  Then,   for all $p_*\in \Omega^*$ and  every $N\in \mathbb{N}$,  the following inequality holds 
\begin{equation}\label{eq:complexityExogenous}
\min \left\{f(p_k) - f^*~:~ k=0, 1,\ldots, N\right\}\leq \tau \frac{d^2(p_0, p_*) + C_{p_*,\kappa}\sum_{k=0}^{N} t_k^2}{2\sum_{k=0}^{N} t_k}.
\end{equation}
\end {theorem}
\begin{proof}
Let $p_*\in \Omega^*$. Since $ \Omega^*  \subset \Omega$,    applying Lemma~\ref{lem:MainIneq} with $q=p_*$, we obtain 
$$
d^2(p_{k+1}, p_*)\leqslant d^2(p_k, p_*) + C_{p_*,\kappa}t_k^2 + 2 \frac{t_k}{\|s_k\|}[f^*-f(p_{k})],  \qquad  s_{k}\in \partial f(p_k), 
$$
for all $k=0,1,\ldots$. Hence, performing the sum of  the above inequality  for $k=0,1, \ldots, N$,  after some algebraic manipulations, we have
$$
 2\sum_{k=0}^{N}   \frac{t_k}{\|s_k\|}[f(p_{k})-f^*] \leq d^2(p_0, p_*)- d^2(p_{N+1}, p_*) + C_{p_*,\kappa}\sum_{k=0}^{N} t_k^2.
$$
Since $f$ is  Lipschitz continuous with  constant $\tau \geq 0$, we have $\|s_k \|\leq \tau$,  for all $k=0,1,\ldots$. Therefore, 
$$
  \frac{2}{\tau} \min\left\{f(p_k) - f^*:~ k=0, 1,\ldots, N\right\}\sum_{k=0}^{N} t_k \leq d^2(p_0, p_*) + C_{p_*,\kappa}\sum_{k=0}^{N} t_k^2, 
$$
which is 	equivalent to the desired inequality.
\end{proof}
\begin{remark}
Note that, for Riemannian manifold with non-negative curvature the inequality in \eqref{eq:complexityExogenous} holds for all $\kappa<0$.  Since   $\lim_{\kappa \to 0}C_{q,\kappa}=1$,  Theorem~\ref{eq:icomp}  is reduced to  \cite[Theorem 3.3]{BentoFerreiraMelo2017}. 
\end{remark}

We remark that in the first part of the next theorem we do not assume that $\Omega^* \neq \varnothing$. Additionally, It is worth to point out that   the second part was first obtained in \cite{ Wang2018}.  Since it is an immediate consequence of the first part and Lemma~\ref{lem:MainIneq}, we decide to include its proof here.

\begin{theorem}\label{teo.Main}
The following equality holds
\begin{equation} \label{eq:linfs}
\liminf_k f(p_k)=f^*.
\end{equation}
In addition, if $\Omega^* \neq \varnothing$ then the sequence $\{p_k\}$ converges to a point $p_*\in\Omega^*$.
\end{theorem}
\begin{proof} 
Assume by contradiction  that $\liminf_k f(p_k)>f^*.$ In this case, we have  $\Omega \neq \varnothing$. Thus, from Lemma~\ref{lem:bounded}, we conclude that $\{p_k\}$ is bounded and,   consequently, by using  Proposition~\ref{prop:subglim},  the sequence $\{s_k\}$ is also bounded. Let  $C_1>0$ such that  $\left\|s_k\right\|<C_1$, for  $k=0,1,\ldots$. On the other hand,  letting $q\in \Omega$, there exist $C_2>0$ and $k_0\in\mathbb{N}$ such that $f(q)<f(p_k)-C_2, $ for all  $k\geq k_0$.  Hence, using Lemma~\ref{lem:MainIneq} and considering that  $\left\|s_k\right\|<C_1$, for  $k=0,1,\ldots$,    we have
$$
d^2(p_{k+1},q)\leq d^2(p_k,q) + C_{q,\kappa}t^2_k-2\frac{C_2}{C_1}t_k, \qquad  k=k_0, k_0+1, \dots.
$$
Consider  $\ell \in \mathbb{N}$. Thus, from the last inequality, after some calculations, we conclude that
$$
\frac{2\,C_2}{C_1}\sum_{j=k_0}^{\ell+k_0}t_j \leq d^2(p_{k_0},q)  - d^2(p_{k_0+ \ell},q)  +  C_{q,\kappa} \sum_{j=k_0}^{\ell+k_0}t^2_j \leq d^2(p_{k_0},q) +  C_{q,\kappa} \sum_{j=k_0}^{\ell+k_0}t^2_j .
$$
Since the last inequality holds for all $\ell \in \mathbb{N}$,   then  using the inequality in \eqref{step-size} we have a contraction. Therefore, \eqref{eq:linfs} holds.

For proving the last statement, let us assume that $\Omega^*\neq\varnothing$. In this case, we have $\Omega\neq\varnothing$ and, from Lemma~\ref{lem:bounded}, the sequence $\{p_k\}$ is bounded. Moreover,  Lemma~\ref{lem:MainIneq} implies, in particular,  that $\{p_k\}$ is  quasi-F\'ejer convergent to $\Omega$.    The equality  \eqref{eq:linfs} implies  that $\{f(p_k)\}$ possesses a decreasing monotonous subsequence $\{f(p_{k_j})\}$ such that
$\lim_{j\rightarrow \infty}f(p_{k_j})= f^*.$ We can assume  that  $\{f(p_k)\}$ is decreasing, monotonous and converges to $f^*$. Being bounded, the sequence $\{p_k\}$ possesses a convergent subsequence $\{p_{k_\ell}\}$. Let us say that
$ \lim_{\ell\rightarrow\infty}p_{k_\ell}=p_*,$ which by the continuity of $f$ implies $f(p_*)=\lim_{\ell\rightarrow\infty}f(p_{k_\ell})=f^*,$
and then $p_*\in\Omega$.  Hence, $\{p_k\}$ has an cluster point $p_*\in\Omega$, and due to $\{p_k\}$ be  quasi-F\'ejer convergent to $\Omega$, it follows from Theorem~\ref{teo.qf} that the sequence $\{p_k\}$ converges to $p_*$. \end{proof}


\subsection{Subgradient Method with Polyak Stepsize}
In this section,  we assume that  $\Omega^* \neq \varnothing$ and  $\{p_k\}$  is generated by   Algorithm~\ref{alg:subgradient}  with  Strategy~\ref{Poliak.Step}  for choosing the step-size.  Let us define
\begin{equation} \label{eq:kd0}
 C_{\kappa, d_0}:=  \frac{2}{\alpha}-\frac{\sqrt{|\kappa|}d_0}{\tanh\left(\sqrt{|\kappa|}d_0\right)}>0, 
\end{equation} 
where $\alpha$ and $d_0$  are  defined in \eqref{StepsizePolyak}.

\begin{remark} \label{re:k0}
Since $\lim_{t\to 0} \tanh (t)/t=1$,  we conclude  that for  Riemannian manifolds with nonnegative curvature, namely, for $\kappa=0$,   \eqref{eq:kd0}  become $ C_{\kappa, d_0}\equiv 2/\alpha -1>0$.
\end{remark}

In the next result, we apply Lemma~\ref {lemli}  to obtain an inequality that plays an important role in our analysis.  Before state this result,  we set 
\begin{equation} \label{eq:ProjOmegaStar}
 {\bar q}\in \Omega^* \quad \mbox{such that} \quad d_0=d(p_0,{\bar q}).
\end{equation}

\begin{lemma}\label{lem:MainIneqPolyak}
Let ${\bar q}\in \Omega^*$ satisfying \eqref{eq:ProjOmegaStar}. Then  the following inequality holds
$$
d^2(p_{k+1}, {\bar q})\leq d^2(p_k, {\bar q}) - C_{\kappa, d_0}  \alpha^2 \frac{\left[f(p_k)-f^*\right]^2}{\left\|s_k\right\|^2}, \qquad k=0,1,\ldots.
$$
\end{lemma}
\begin{proof}
First we are going to prove that $d(p_k, {\bar q})\leq d_0$,  for all $k=0,1,\ldots$. The proof will be made by induction. For $k=0$ is immediate. Assume that $d(p_k, {\bar q})\leq d_0$. Using the  second inequality of Lemma~\ref{lemli} with $q={\bar q}$,  $t=t_k$, $p=p_k$, $s=s_k$,  $p_{k+1}=\gamma(t_k)$ and  considering  that $f^*=f({\bar q})$, we obtain
$$
d^2(p_{k+1}, {\bar q})\leq d^2(p_k, {\bar q}) +
 \frac{\sinh\left(\sqrt{|\kappa|}t_k\right)}{\sqrt{|\kappa|}t_k}\left[\frac{\sqrt{|\kappa|}d(p_k, {\bar q})}{\tanh\left(\sqrt{|\kappa|}d(p_k, {\bar q})\right)} t_k^2+
\frac{2t_k}{\|s_k\|}\left[f^*-f(p_k)\right]\right]. 
$$
Since the map $(0, +\infty) \ni t\mapsto t/\tanh (t)$ is increasing, using the assumption  $d(p_k, {\bar q})\leq d_0$ and  definition of $t_k$ in  \eqref{StepsizePolyak}, the last inequality becomes
\begin{equation}\label{des.step.poly}
d^2(p_{k+1}, {\bar q})\leq d^2(p_k, {\bar q}) +
 \frac{\sinh\left(\sqrt{|\kappa|}t_k\right)}{\sqrt{|\kappa|}t_k}\left[\,\frac{\sqrt{|\kappa|}d_0}{\tanh\left(\sqrt{|\kappa|}d_0\right)}-\frac{2}{\alpha}\right] \alpha^2 \frac{\left[f(p_k)-f^*\right]^2}{\left\|s_k\right\|^2}.
\end{equation}
Thus, the inequalities  in  \eqref{StepsizePolyak} imply that   $d(p_{k+1}, {\bar q})\leq d(p_k, {\bar q})\leq d_0$ and the induction is concluded.  Hence, $d(p_k, {\bar q})\leq d_0$,  for all $k=0,1,\ldots$.  Therefore,   we can also prove that  \eqref{des.step.poly} hods,  for all $k=0,1,\ldots$.  Taking into account that  $\sinh(\sqrt{|\kappa|}t_k)/(\sqrt{|\kappa|}t_k)\geq 1$, the combination of  second inequality in \eqref{StepsizePolyak}, \eqref{eq:kd0} and  \eqref{des.step.poly}  yield   the desired inequality. 
\end{proof}

\begin{remark}\label{poliak.ineq.kz}
Since $\lim_{t\to 0} \tanh (t)/t=1$ and  $\lim_{t\to 0} \sinh (t)/t=1$,  then by using similar idea  considered in the proof  of Lemma~\ref{lem:MainIneqPolyak}, we can show that, for  Riemannian manifolds with nonnegative curvature,  holds $d^2(p_{k+1}, q)\leq d^2(p_k, q) - (2/\alpha -1) t^2_k$, for all $k=0,1,\ldots$ and all $q\in \Omega^*$.
\end{remark}

The next result presents an iteration-complexity bound for the subgradient method with the Polyak's step-size rule.
\begin{theorem}\label{teo.complexity.polyak}
Assume that     $f: \mathcal{M}\to  \overline{\mathbb{R}}$  is   Lipschitz continuous with  constant $\tau \geq 0$. 
Let ${\bar q}\in \Omega^*$ satisfying \eqref{eq:ProjOmegaStar}. Then,   for every $N\in \mathbb{N}$,  there  holds  
\begin{equation} \label{eq:inc}
 \sum_{k=0}^{N} [f(p_{k})-f^*]^2\leq \frac{\tau^2d^2(p_0, {\bar q})}{{ C_{\kappa, d_0}}}.
\end{equation}
As a consequence, 
\begin{equation}\label{eq:complexityPolyak}
\min\left\{f(p_k) - f^*~:~ k=0, 1,\ldots, N\right\}\leq [\tau d(p_0, {\bar q})]/\sqrt{{ C_{\kappa, d_0}}(N+1)}.
\end{equation}
\end {theorem}
\begin{proof}
Since $f$ is  Lipschitz continuous with  constant $\tau \geq 0$, we have $\|s_k \|\leq \tau$,  for all $k=0,1,\ldots$. Thus, it follows from   Lemma~\ref{lem:MainIneqPolyak}   that 
$$
[f(p_k) - f^*]^2 \leq \frac{\tau^2}{ C_{\kappa, d_0}\alpha^2}[d^2(p_k, {\bar q}) - d^2(p_{k+1}, {\bar q})], \qquad k=0,1,\ldots.
$$
Performing the sum of  the above inequality  for $k=0,1, \ldots, N$, we obtain  \eqref{eq:inc}.  The second statement of the theorem is an immediate consequence of the first one.
\end{proof}

\begin{remark}
It is  worth noting that if $\kappa=0$ we have $C_{q,\kappa}=1$ and then  Theorem~\ref{teo.complexity.polyak}  merges into  the inequality \cite[Theorem 3.4]{BentoFerreiraMelo2017}. 
\end{remark}

\begin{theorem}\label{teo.MainPolyak}
The following equality holds $\lim_{k\rightarrow \infty} f(p_k)=f^*$. Consequently, all cluster point of   $\{p_k\}$ is a solution of \eqref{eq:OptP}.
\end{theorem}
\begin{proof}
Letting  $N$ goes to $+\infty$ in  \eqref{eq:inc},  we conclude that  $\lim_{k\rightarrow \infty} f(p_k)=f^*$.  It follows from  Lemma~\ref{lem:MainIneqPolyak} that  $\{p_k\}$ is bound.  For concluding the proof, let ${\bar p}$   accumulation point of   $\{p_k\}$ and   $\{p_{k_i}\}$ a subsequence  of   $\{p_k\}$ such that  $\lim_{k_i\to + \infty}p_{k_i}={\bar p}$.  Therefore, $f({\bar p})=\lim_{k_i\rightarrow \infty} f(p_{k_i})=f^*$ and then ${\bar p}\in \Omega^*$. 
\end{proof}

\begin{corollary}
For $\kappa=0$ the sequence $\{p_k\}$ converges  to a point $q\in \Omega^*$.
\end{corollary}
\begin{proof}
Lemma~\ref{lem:MainIneqPolyak} implies that $\{p_k\}$ is bounded. As a consequence, $\{p_k\}$ has at least one cluster point. Thus, by Theorem~\ref{teo.MainPolyak}, it follows that there exists a subsequence $\{p_{k_i}\}$ of $\{p_k\}$ converging to a $q\in\Omega^*$. Hence, $ \lim_{k_i\rightarrow \infty}  d(p_{k_i},q)=0$. Since the inequality of Remark~\ref{poliak.ineq.kz} implies that $\{d(p_k,q)\}$ is monotonic decreasing, it holds  that $\lim_{k\rightarrow \infty}  d(p_{k},q)= 0$, completing the proof.
\end{proof}

\section{Numerical examples}\label{sec:num.exem}

In this section, we numerically illustrate the results on complexity-iteration bounds of Section~\ref{sec:subgradient}. For this aim, we consider the convex feasibility problem in Riemannian setting which consists of finding a point $p\in \mathcal{M}$  such that 
\begin{equation} \label{eq:CFP}
 p\in C:=\bigcap_{i=1}^{m}C_i, \qquad \qquad   C_i:=\{p\in \mathcal{M}: \,f_i(p)\leq 0\},
\end{equation}
 where $f_i:\mathcal{M}\rightarrow\mathbb{R}$ is  convex, for all $i=1,...,m$. This problem can be equivalently rewritten as an optimization problem~\eqref{eq:OptP} where  $f:\mathcal{M}\rightarrow \mathbb{R}$ is given by 
$$f(p):=\max\left\{f_1(p),\ldots,f_m(p),0\right\}.$$
Note that $f(x)\geq 0$ for all $x\in \mathcal{M}$. If $C\neq\varnothing$, then  $C=\{p\in \mathcal{M}:~f(p)=0\}$. Thus, $C$ is the solution set of the problem \eqref{eq:OptP}  and $f^*=0$. Now, if the interior of $C$  is nonempty, i.e, $\mbox{int}C \neq\varnothing$, 
then there exist $\epsilon>0$ and $\hat{x}\in \mathcal{M}$ such that $f_i(\hat{x})\leq -\epsilon$, for all  $i=1,\ldots,m$. In this case, defining 
\begin{equation} \label{eq:CFP.func.slater}
f(p):=\max\left\{f_1(p),\ldots,f_m(p),-\epsilon\right\}, 
\end{equation}
 the solution set of the problem \eqref{eq:OptP} is contained in $\mbox{int}C$ and $f^*=-\epsilon$. 
 
Our examples consist of convex feasibility problems~\eqref{eq:CFP} where $C$ has nonempty interior. Let us explain how the examples were generated. Let $ \mathcal{M}$ be a Riemannian Manifold with sectional curvature bounded from above by $K$ and set
$$\rho_{K}:=\frac{1}{2}\min\left\{\mbox{inj}\,\mathcal{M},~\frac{\pi}{2\sqrt{K}}\right\},$$ 
where $\mbox{inj}\,\mathcal{M}$ is the injectivity radius of $\mathcal{M}$, with the convention that $1/\sqrt{K}=+\infty$ for $K\leq 0$; see \cite[pag. 110]{Sakai1996}. Let $d$ be the associated Riemannian distance. Set $q\in\mathcal{M}$, and choose $r>0$ and $v_1,...,v_m\in T_q\mathcal{M}$ in such a way that
\begin{equation} \label{eq:ai}
a_i:= \exp_{q}\left(r\frac{v_i}{\left\|v_i\right\|}\right) \in B(q,\rho),
\end{equation}
for all $i=1,\ldots,m$, and some $\rho<\rho_K$. Since $d(a_i, q)=r$,  we conclude that   $a_i\in \partial B[q,r]$, where $\partial B[q,r]$ denotes the boundary of $B[q,r]$, for all $i=1,\ldots,m$. Let $\epsilon>0$ and define $f_i: \mathcal{M}\to \mathbb{R}$ by 
$$
f_i(p):=d(p,a_i)-r-\epsilon, 
$$
for each $i=1,\ldots,m$, and consider $f$ given by~\eqref{eq:CFP.func.slater}. In this case, we have  $B[q,\epsilon]\subset C$, $f^*=-\epsilon$, and $d_0\leq d(p_0,q)$ where $d_0$ is defined in~\eqref{StepsizePolyak}. Moreover, $f$ is Lipschitz continuous with constant $\tau = 1$. Given $p\in \mathcal{M}$, it follows that
$$-\sum_{j\in I_p}\alpha_j\frac{\exp^{-1}_p a_j}{d(a_j,p)}  \in  \partial f(p),$$
where $I_p:=\left\{j:p\neq a_j, \,\, j=1,\ldots, m\right\}$ and $\sum_{j\in I_{p}}\alpha_j=1$, see~\cite{BentoFerreiraOliveira2015,Yang2010}. We generated two examples with different types of Riemannian manifolds $\mathcal{M}$ as described below.

\begin{example}[Positive definite  symmetric matrices] \label{ex:psm}
Let $\mathbb{P}^n$ and ${\mathbb P}^n_{++}$ be the set of symmetric  matrices and the set of positive definite  symmetric matrices, respectively.   Let $\mathcal{M}\coloneqq ({\mathbb P}^n_{++}, \langle \cdot , \cdot \rangle)$ be the Riemannian manifold endowed with the Riemannian metric given by  
\begin{equation*}\label{eq:metric}
\langle U,V \rangle\coloneqq \mbox{tr} (VX^{-1}UX^{-1}),\qquad X\in \mathcal{M}, \qquad U,V\in
T_X\mathcal{M}\approx\mathbb{P}^n,
\end{equation*}
where $\mbox{tr}(X)$ denotes the trace of  $X$.  We remark that  $\mathcal{M}$ is a Hadamard manifold;  see,   for example,  \cite[Theorem 1.2. p. 325]{Lang1999} and  its curvature  is bound below;  see \cite{LengletRoussonDericheFaugeras2006}. The  exponential mapping   $\exp_X:T_X\mathcal{M}\to \mathcal{M}$ and its inverse   $\exp^{-1}_X: \mathcal{M}\to T_x\mathcal{M}$ in  ${\mathcal M}$ are given, respectively,  by
\begin{equation*} \label{eq:ExpInv}
\exp_XV:=X^{1/2}e^{X^{-1/2}VX^{-1/2}}X^{1/2}, \qquad \exp^{-1}_XY:=X^{1/2}\ln(X^{-1/2}YX^{-1/2})X^{1/2}.
\end{equation*}
Denotes by  $\left\|\cdot\right\|_F$ the Frobenius norm associated to the  inner product  $\langle U,V \rangle_F\coloneqq \emph{tr} (VU)$, for all $U,V\in {\mathbb P}^n_{++}$. Let   $d$ be the Riemannian distance defined in $\mathcal{M}\coloneqq ({\mathbb P}^n_{++}, \langle \cdot , \cdot \rangle)$, i.e.,  
$$
d(A, X)=\left\|\ln\left(X^{-1/2}AX^{-1/2}\right)\right\|_F, \qquad A,X\in {\mathbb P}^n_{++},
$$
see \cite{NesterovTodd2002}. 

We set $n=10$, $m=10$, $r=1$, and $\epsilon=0.1$. We random generated matrix $q\in {\mathbb P}^n_{++}$ and the starting point $p_0\in {\mathbb P}^n_{++}$ with eigenvalues belonging to $(0,100)$, and matrices $v_1,...,v_m\in \mathbb{P}^n$ with eigenvalues belonging to $(-100,100)$. Then, matrices $a_1,...,a_m\in {\mathbb P}^n_{++}$ were generated according to~\eqref{eq:ai}.
\end{example}
 \begin{example}[Sphere]   \label{ex:sphere}
Let $\mathbb{S}:=\{ x\in \mathbb{R}^n:~\|x\|=1\}$ be the  $(n-1)$-dimensional unit sphere.   Endowing the sphere $\mathbb{S}$ with  the Euclidean metric $\left\langle \cdot,\cdot\right\rangle$   we obtain   a  complete Riemannian manifold with    curvature equal to $1$, which will be also  denoted by   $\mathbb{S}$. The tangent plane at $x\in\mathbb{S}$ is given by  $T_x\mathbb{S}:=\{v\in \mathbb{R}^n : ~\langle v, x\rangle=0\}$ and the  exponential mapping $\exp_x:T_x\mathcal{M}\to \mathcal{M}$ in  assigned by 
\begin{equation*}\label{expsphera}
\exp_xv:=
\begin{cases}
\cos(\left\|v\right\|)x+\sin(\left\|v\right\|)\frac{v}{\left\|v\right\|},\qquad v\neq 0, \\
x,\qquad \qquad \qquad \qquad  \qquad  \qquad  v=0.
\end{cases}
\end{equation*}
The inverse of the  exponential mapping $\exp^{-1}_x:\mathcal{M}\to T_x\mathcal{M}$ is given by 
$$ 
\exp_x^{-1}y := \frac{\arccos \left\langle x,y\right\rangle}{\sqrt{1-\left\langle x,y\right\rangle^2}}(I-xx^T)y,
$$
The Riemannian distance   between  $x,y\in\mathbb{S}$ is given by  $d(x,y)= \arccos \left\langle x,y\right\rangle$, for more  details; see, for example, \cite{FerreiraIusemNemeth2014}. 

We set $n = 200$, $m = 50$, $r=\pi/16$, and $\epsilon=0.001$. We defined $q=(1,\ldots,1)/\sqrt{n}$ and random generated vectors $v_1,...,v_m\in T_q\mathbb{S}$. Then, vectors $a_1,...,a_m\in \mathbb{S}$ were generated according to~\eqref{eq:ai}. The starting point $p_0\in \mathbb{S}$ was generated by taking a random vector $v\in T_q\mathbb{S}$ and setting $p_0= \exp_{q}\left(\lambda \frac{\pi}{8}\frac{v}{\left\|v\right\|}\right)$, where $\lambda\in(0,1)$.
\end{example}
We coded Algorithm~\ref{alg:subgradient} in Matlab and run it on the above examples. For Example~\ref{ex:psm} we used the exogenous step-size give by $t_k=1/(k+1)$ for all $k=0,1,\ldots$, while for Example~\ref{ex:sphere} we adopted the Polyak's step-size with $\alpha = 1.9999\times \tanh(d(p_0,q))/d(p_0,q)$. For each example, since $f^*=-\epsilon$, by Theorems~\ref{teo.Main} and \ref{teo.MainPolyak} respectively, there exists $k_0$ such that $p_k\in C$ for all $k\geq k_0$. Therefore, these convex feasibility problems are solved by Algorithm~\ref{alg:subgradient} in a finite number of iterations. Indeed, Algorithm~\ref{alg:subgradient} found a feasible point with 55 and 41 iterations for Examples~\ref{ex:psm} and \ref{ex:sphere}, respectively. Figure~\ref{fig:complexity}~(a) corresponds to Example~\ref{ex:psm} and reports the function values of the left and right hand sides of inequality~\eqref{eq:complexityExogenous} for each iteration of Algorithm~\ref{alg:subgradient}. In its turn, Figure~\ref{fig:complexity}~(b) is related to Example~\ref{ex:sphere} and illustrates the iteration-complexity bound given by~\eqref{eq:complexityPolyak}.

 \begin{figure}[h!]
\begin{minipage}[b]{0.50\linewidth}
\begin{figure}[H]
	\centering
		\includegraphics[scale=0.55]{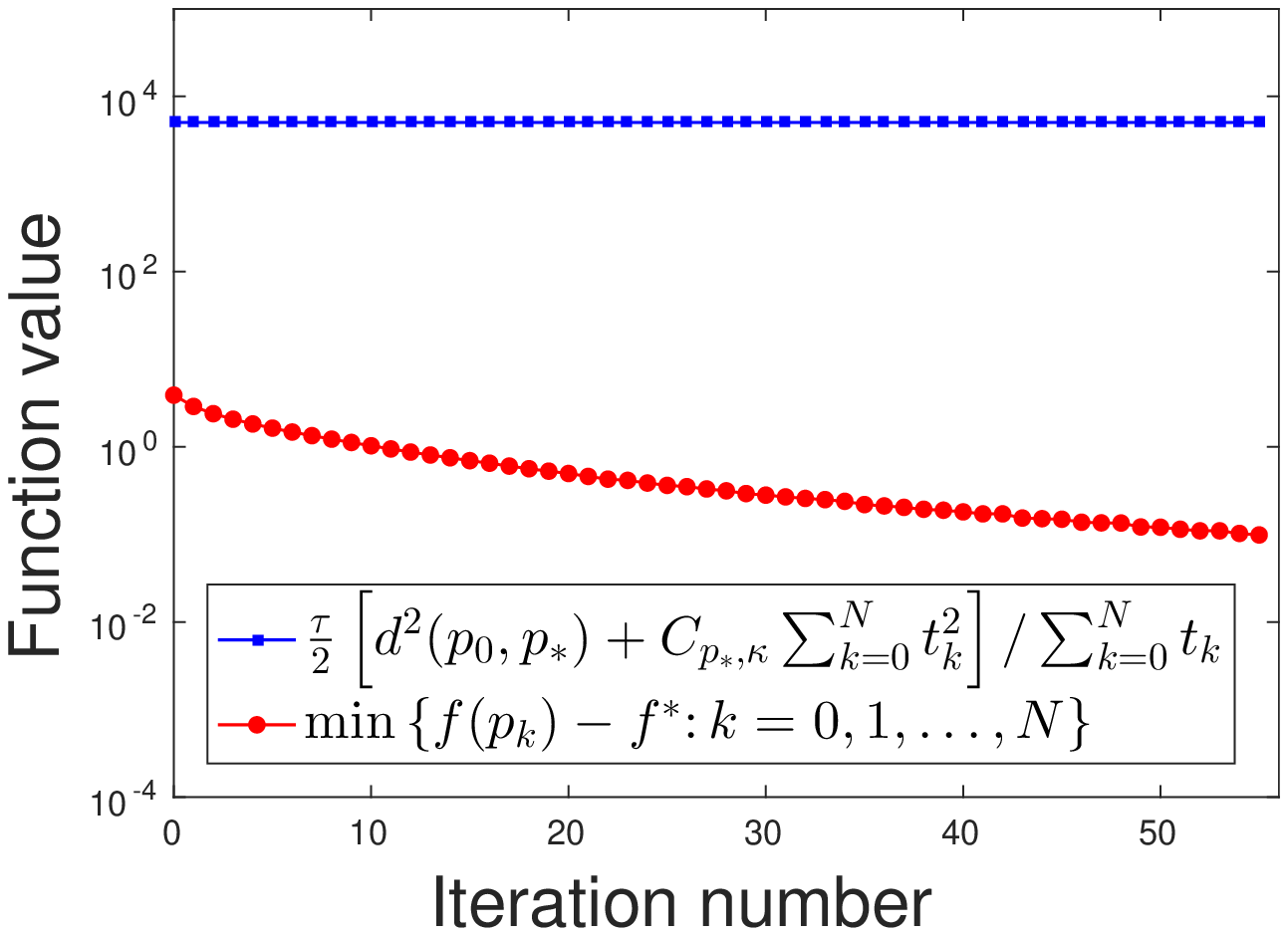}\\
	\footnotesize	(a) 
\end{figure}

\end{minipage} \hfill
\begin{minipage}[b]{0.50\linewidth}

\begin{figure}[H]
	\centering
		\includegraphics[scale=0.55]{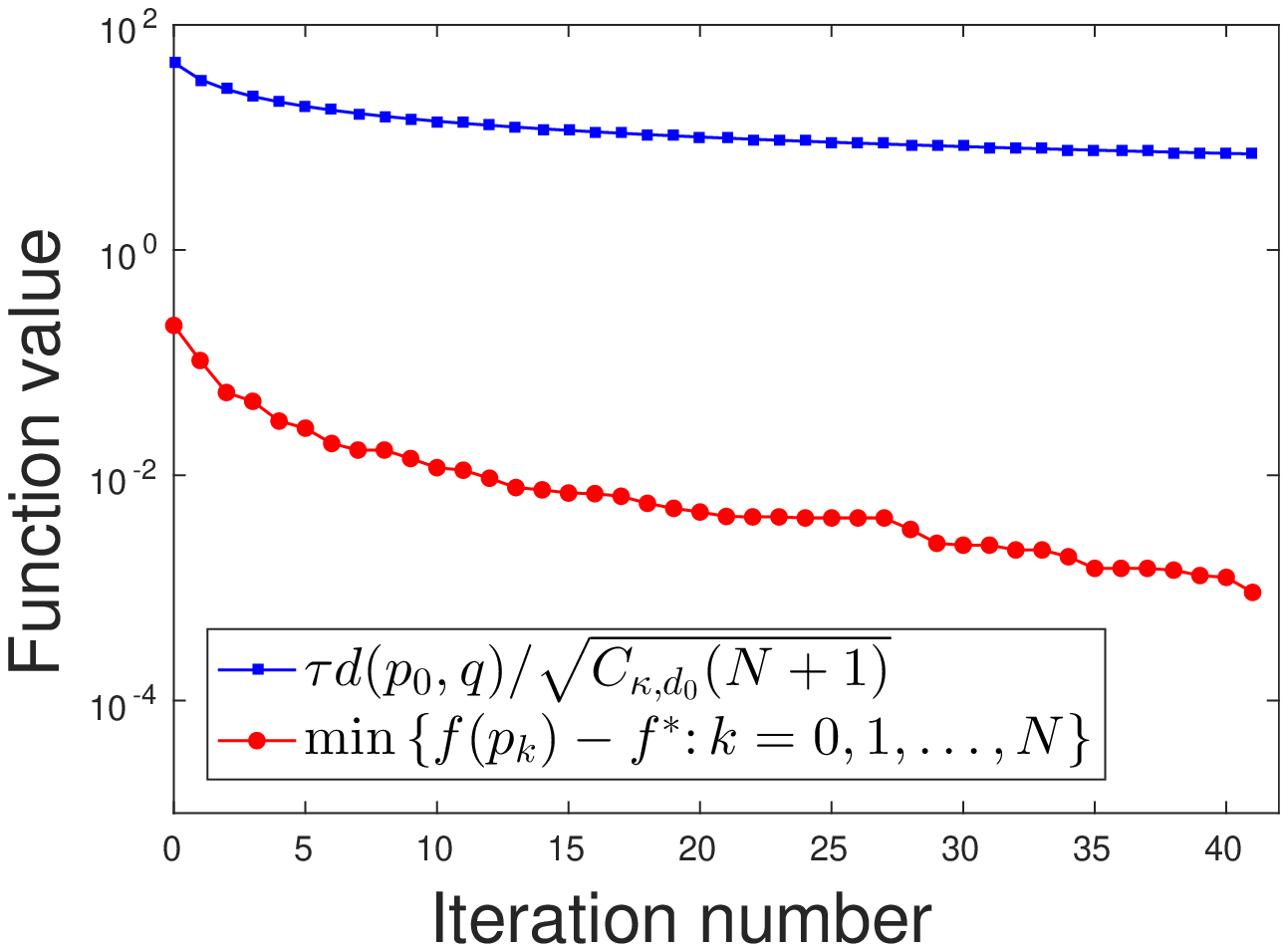}\\
	\footnotesize	(b)
\end{figure}
\end{minipage}\hfill
\caption{Iteration-complexity bound for the Riemannian subgradient method with: (a) exogenous step-size applied to Example~\ref{ex:psm} -- Theorem~\ref{eq:icomp}; (b) Polyak's step-size applied to Example~\ref{ex:sphere} -- Theorem~\ref{teo.complexity.polyak}.}
\label{fig:complexity}
\end{figure}

 As can be seen in Figure~\ref{fig:complexity}, inequalities~\eqref{eq:complexityExogenous} and \eqref{eq:complexityPolyak} in Theorems~\ref{teo.Main} and \ref{teo.MainPolyak} are met for all iterations of Algorithm~\ref{alg:subgradient}, illustrating the practical reliability of our iteration-complexity results.
\section{Conclusion}\label{sec.conclu}
In this paper, we analyzed the iteration-complexity of subgradient method with exogenous step-size and Polyak's step-size. In general, the Polyak's step-size has a better performance than the exogenous step-size, but the choice of exogenous step-size is also interesting
because it does not depend on any data computed during the algorithm, being important in large scale optimization problems.
Since the  feasibility and  optimization problems are closed related,  this paper complements the understanding of the subgradient algorithm in this  settings.  Finally, we remark that for Riemannian manifolds with  curvature  unbounded below,  perhaps another strategy for the step  will be need  since   is not possible to control the distance between the geodesics. Indeed, if the curvature is positive the geodesics  emanating from the same point tend to approximate one each other, the contrary occurs if the curvature is negative. 

\def\cprime{$'$} \def\cprime{$'$} \def\cprime{$'$}

\end{document}